\documentclass{article}

\author{Eric Foxall}

\usepackage{amsmath}
\usepackage[margin=1in]{geometry}
\usepackage{graphicx}
\usepackage{comment}
\usepackage{amssymb}
\usepackage{bbm}
\usepackage{mathrsfs}

\usepackage{amsthm}
\usepackage{amsmath,amssymb}

\DeclareMathOperator{\argmin}{argmin}
\DeclareMathOperator{\base}{base}
\DeclareMathOperator{\intr}{int}
\DeclareMathOperator{\dist}{d}

\numberwithin{equation}{section}
\newtheorem{definition}{Definition}[section]
\newtheorem{theorem}{Theorem}[section]

\newtheorem{lemma}{Lemma}[section]

\title{The SEIS model, or, the \\contact process with a latent stage}

\begin{document}
\maketitle



\begin{abstract}
The susceptible-exposed-infectious-susceptible (SEIS) model is well-known in mathematical epidemiology as a model of infection in which there is a latent period between the moment of infection and the onset of infectiousness.  The compartment model is well studied, but the corresponding particle system has so far received no attention.  For the particle system model in one spatial dimension, we give upper and lower bounds on the critical values, prove convergence of critical values in the limit of small and large latent time, and identify a limiting process to which the SEIS model converges in the limit of large latent time.
\end{abstract}

\noindent\textbf{Keywords: }SEIS model, Contact process, Interacting Particle Systems\\
\textbf{MSC 2010: }60J25, 92B99\\


\section{Introduction}\label{secintro}
The SEIS model is a model of the spread of an infection that in addition to the usual susceptible and infectious classes includes an exposed class that is infected but not yet infectious; it can be used to model infections such as gonorrhea in which there is a short latent stage before the onset of infectiousness, and in which recovery from the infection confers no immunity.  The classical model, usually called a compartment model, is deterministic and consists of a set of three differential equations describing the evolution of the number of susceptible, exposed and infectious individuals (the three compartments), which for simplicity are taken to be real-valued; for a formal definition see \cite{epid}, Chapter 2.  The model has either a globally stable disease-free state or an unstable disease-free state together with a globally stable endemic state, according as the basic reproduction number for the infection is $\leq 1$ or $>1$; see \cite{koro} for a proof using Lyapunov functions.\\

Now, the classical SEIS model is deterministic and assumes that the population is well-mixed.  However, there is a natural way to define an SEIS model that incorporates both spatial and random effects, using an interacting particle system (see \cite{ips} for an introduction to interacting particle systems).  For the simpler SIS model with no exposed class, this system is called the contact process, and has been well studied over the last forty years, in a variety of spatial settings including the $d$-dimensional integer lattices, trees, random graphs, and even more general sequences of finite graphs; \cite{ips} and \cite{sis} give an overview of results up to 1985 and 1999 respectively, and \cite{rgd} includes a survey of results on random graphs up to about 2009; a recent result on fairly general sequences of finite graphs can be found in \cite{exp}.\\

For the contact process on the $d$-dimensional lattice $\mathbb{Z}^d$ with a single initially infectious site there is a critical value $\lambda_c$ of the infection parameter $\lambda$ such that for $\lambda\leq\lambda_c$ the process dies out with probability 1, and for $\lambda > \lambda_c$ the process survives with a positive probability, spreading linearly in time and converging to a non-trivial invariant measure when it survives; for a proof of convergence see \cite{sis}, Part I, and for a proof of linear spread in $d=1$ and linear spread with convergence to a limiting shape in $d \geq 2$, see \cite{speed} and \cite{shape}.\\

According to numerical simulations in $d=1,2$, the SEIS model behaves in the same way as the contact process, in this case with a critical value that varies slightly with the average latent time, and spreading linearly in time when it survives.  However, for the SEIS model it is not clear how to prove this, because of the absence of a property called monotonicity that enables much of the analysis of the contact process.  Nevertheless, we can show that the infection survives when the infection parameter is large enough, uniformly in the latent time, and we can obtain reasonable bounds when the latent time is either very large or very small.  In addition, in the limit of large latent time the model, when properly rescaled in time, approaches a limit process, and we describe the limit process and the convergence to the limit process.  We begin by describing the process and summarizing the main results.\\

\section{Main Results}
To distinguish it from the compartment model, we use ``SEIS process'' to refer to the SEIS model as an interacting particle system.  Given a finite or countably infinite connected undirected graph $G=(V,E)$ with bounded degree i.e., for some $d<\infty$, $|\{y \in V:xy \in E\}| \leq d$ for each $x \in V$, the SEIS process with infection parameter $\lambda >0$ and average latent time $\tau \geq 0$ is defined as follows.  Letting $0$ denote susceptible, $1$ denote exposed and $2$ denote infectious, each site $x \in V$ is in one of the states $0,1$ or $2$ (that we later refer to as \emph{types}), with transitions
\begin{itemize}
\item $0\rightarrow 1$ at rate $\lambda n_2(x)$ (transmission)
\item $1\rightarrow 2$ at rate $1/\tau$ or instantaneously if $\tau=0$ (onset)
\item $2\rightarrow 0$ at rate 1 (recovery)
\end{itemize}
where $n_2(x)$ is cardinality of the set $\{xy\in E:y \textrm{ is in state 2 }\}$.  The case $\tau=0$ is the contact process with transmission parameter $\lambda$.  The meaning of ``rate'' is that in the absence of other transitions, each transition occurs after an amount of time which is exponentially distributed with parameter given by the rate.\\

A standard reference on particle systems, and methods for constructing them, can be found in \cite{ips}.  Since it will help us later on, we follow \cite{gc} and use a graphical representation to construct the process.  We begin with the \emph{spacetime set} $\mathcal{S} = G\times [0,\infty)$, which we picture as a copy of $G$ extruded upward along fibers in the increasing time direction; this is particularly easy to imagine if $G$ is a planar graph.  When required we use the topology on $\mathcal{S}$ with base $\{a\}\times (t,t'):a \in V \cup E, t<t'\}$.  Place independent 1-dimensional Poisson point processes (p.p.p.'s) along fibers $\{\cdot\}\times [0,\infty)$ as follows:
\begin{itemize}
\item at each site $x \in V$, recovery with intensity 1 and label $\times$,
\item at each site $x \in V$, onset with intensity $1/\tau$ and label $\star$ if $\tau>0$, or omitting if $\tau=0$, and
\item along each edge $xy \in E$, transmission with intensity $\lambda$ and label $\leftrightarrow$.
\end{itemize}
This furnishes the probability space $\Omega$, which we can think of as a random labelling of $\mathcal{S}$, and which we refer to as the \emph{substructure}; the notation $\mathbb{P}$ is used to denote the law of the substructure, and when necessary, we write for example $\mathbb{P}_{\lambda}$ or $\mathbb{P}_{\tau}$ to emphasize the dependence on parameters.  Define the following notation: for a Borel measurable set $R \subset \mathcal{S}$ let $\mathcal{F}(R)$ denote the $\sigma$-algebra generated by the restriction of the substructure to $R$, and for $t>0$ let $\mathcal{F}(t)=\mathcal{F}(G\times [0,t])$.  If we view the substructure as a function of time $U_t$ then $U_t$ is adapted to the filtration $\mathcal{F}(t)$ and it follows from the strong Markov property applied to $U_t$ that for any stopping time $s$ the law of $\mathbb{E}(U_{t+s}|\mathcal{F}(s))$ is the same as the law of $U_t$.  Also, it follows from standard properties of p.p.p's that if $\{R_i:i=1,2,...\}$ are pairwise disjoint then $\{\mathcal{F}(R_i):i=1,2,...\}$ are independent; the same is true if the sets are disjoint up to measure zero in the measure on $\mathcal{S}$ given by the product of counting measure on edges and vertices of $G$ with Lebesgue measure on $[0,\infty)$.  Both of these facts will be useful throughout the paper.\\

Given an initial configuration $\eta_0 \in \{0,1,2\}^V$, to determine $\eta_t(x)$ for each realization $\omega \in \Omega$ consider the set $T_t(x)$ of points $(u,s) \in G\times [0,t]$ that can reach $(x,t)$ by moving either upwards in time along vertices or horizontally along transmission labels $\leftrightarrow$.  In order to compute $\eta_t(x)$ from the transition labels it suffices to compute $\eta_s(y)$ for $(y,s) \in T_t(x)$.  By a simple comparison, $|\{u\in G:(u,t-s) \in T_t(x)\}|$ is bounded above by a branching process with no deaths in which pairs of offspring are produced at rate $\lambda d$, so with probability $1$, $T_t(x)$ is a bounded set, and it follows easily (for example, by considering the events $T_t(x) \subset G_n\times [0,t]$ for a sequence of graphs $G_n$ with $\cup G_n = G$) that the number of labels in $T_t(x)$ is almost surely finite.  Denote the timing of labels by $t_1<t_2<...<t_m$, then given $\eta_{t_i}(x)$ for $x$ such that $(x,t_i)\in T_t(x)$, if the label at time $t_{i+1}$ is
\begin{itemize}
\item $\times$ at $x$ and $\eta_{t_i}(x)=2$ then $\eta_{t_{i+1}}(x)=0$,
\item $\star$ at $x$ and $\eta_{t_i}(x)=1$ then $\eta_{t_{i+1}}(x)=2$,
\item $\leftrightarrow$ along $xy$ and $\eta_{t_i}(x)=0$, $\eta_{t_i}(y)=2$ then $\eta_{t_{i+1}}(x)=1$ if $\tau>0$ and $\eta_{t_{i+1}}(x)=2$ if $\tau=0$.
\end{itemize}
otherwise nothing happens.  Then, let $\eta_t(x) = \eta_{t_m}(x)$.  The reader may easily verify that this approach defines $\eta_t(x)$ for all $x \in V$, $t\geq 0$ in a consistent manner.  A depiction of this construction is given in Figure \ref{graphfig}.  For what follows, say that $x$ is \emph{active} at time $t$ if $\eta_t(x)\neq 0$.  Letting $\mathcal{C}_0$ denote the set of configurations having only finitely many active sites, if $\eta_0 \in \mathcal{C}_0$ then bounding the number of active sites by a branching process in which each particle produces offspring at rate $\lambda$ it follows that $\eta_t \in \mathcal{C}_0$ for $t>0$, and $\eta_t$ behaves like a continuous time Markov chain on $\mathcal{C}_0$ in the sense of \cite{norris}, with transition rates as specified in the description of the model.\\

\begin{figure}
\centering{\includegraphics[width=120mm,height=83mm]{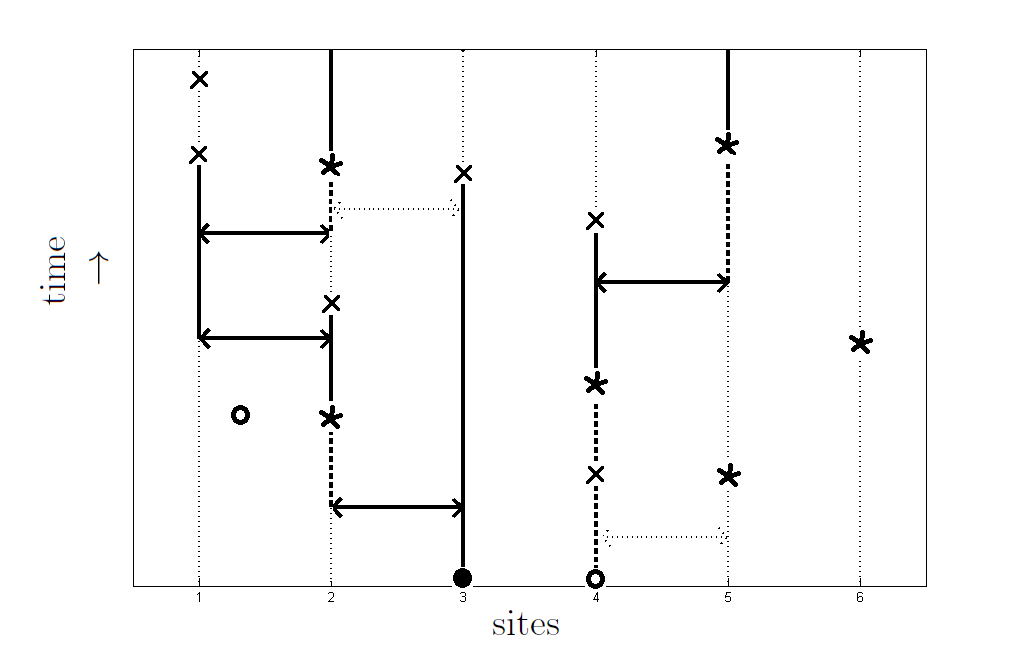}}
\caption{Depiction of the graphical construction for the SEIS model, starting from site 4 exposed and site 3 infectious.  Bold dashed lines denote exposed sites, and bold solid lines denote infectious sites.}
\label{graphfig}
\end{figure}

The above graphical representation supplies a natural coupling of the process for all choices of $\eta_0$, namely the one in which, for each $\eta_0$, $\eta_t$ is determined from $\eta_0$ for $t>0$ via the substructure.  With respect to this coupling, the reader may verify that the contact process, which is the case $\tau=0$, is \emph{monotone} in the partial order $\eta \leq \eta' \Leftrightarrow \forall x, \eta(x)\leq \eta'(x)$ in the sense that $\eta_0\leq\eta_0'$ implies $\eta_t\leq\eta_t'$ for $t>0$.  In fact, the process is also monotone with respect to $\lambda$:  if $\lambda<\lambda'$ then we can couple processes $\eta_t$ with parameter $\lambda$ and $\eta_t'$ with parameter $\lambda'$ so that $\eta_0 \leq \eta_0'$ implies $\eta_t \leq \eta_t'$ for $t>0$.  To do so, for transmission events place independent p.p.p.'s along each edge
\begin{itemize}
\item with intensity $\lambda$ and label $\leftrightarrow$ 
\item with intensity $\lambda'-\lambda$ and label $\leftrightarrow '$
\end{itemize}
with recovery events as before, and for transmission events, for $\eta_t$ use only the labels $\leftrightarrow$, while for $\eta_t'$ use both the labels $\leftrightarrow$ and $\leftrightarrow '$.  Using this fact, and the fact that any configuration with a positive and finite number of active sites can reach any other such configuration, it follows directly that there is a critical value of the transmission parameter that we denote $\lambda_c^0$ (which may a priori be equal to $0$ or $\infty$) such that the infection survives with positive probability when $\lambda>\lambda_c^0$ and $|\eta_0|\geq 1$ ($|\eta|$ denotes the number of active sites in $\eta$) and dies out with probability 1 when $\lambda<\lambda_c^0$ and $|\eta_0|<\infty$, where \emph{survival} means $|\eta_t|>0\,\,\forall t$ and \emph{dying out} means $|\eta_t|=0$ for $t$ large enough.   Whenever we refer to critical values in what follows, they will have this property; the only exceptions are the upper and lower critical values for the SEIS process defined below, for which the above property is split between the two.\\

Given that $\tau=0$ gives the contact process it is natural to ask whether we obtain something as $\tau\rightarrow\infty$.  The answer is yes, if we rescale time so that onset occurs at rate $1$.  We first describe the limit process, then state the sense in which the SEIS process converges to it.  The limit process has the state space $\{0,1\}^V$ where $1$ can be thought of as occupied and $0$ as vacant.  It is defined using the \emph{dispersal distributions} $p(x,\cdot)$ given by letting $p(x,A)$ be equal to the probability that for the contact process with the single infectious site $x$, transmission from $x$ to every site in $A$ occurs, followed by recovery at $x$, without transmission to any sites in $A^c$, and ignoring subsequent transmissions from other newly infected sites.  Each occupied site $x$ becomes vacant at rate $1$, at which point, with probability $p(x,A)$ all the vacant sites in $A$ become occupied.\\
	
There is an obvious graphical representation of the limit process:  at each site place a p.p.p. with intensity $1$ and label $\star$, and at each occurrence of $\star$ at site $x$ sample the dispersal distribution $p(x,\cdot)$, placing a $\rightarrow$ label from $x$ to $y$ for each $y$ to which $x$ disperses, and let the samples be independent.  The rest of the construction follows the same pattern so we omit the details.  It is easy to see the limit process is monotone, and is also monotone in $\lambda$; to see the latter property, for $\lambda<\lambda'$ make a joint construction by coupling dispersal distributions in the obvious way.  Thus the limit process has a critical value that we denote $\lambda_c^{\infty}$ such that the same dichotomy holds as for the contact process above.  The following result describes convergence of the SEIS process to the limit process.\\

\begin{theorem}\label{thmlimit}
For fixed $\lambda$, let $\eta_t$ denote the SEIS process on a countable graph with bounded degree, under the rescaling $t\mapsto t/\tau$, and let $\zeta_t$ denote the limit process.  Let $S = \{t:\eta_t(x)=2 \textrm{ for some }x\}$ denote the set of times when the rescaled SEIS process has an infectious site.  Fix $T>0$ and an initial state with no infectious sites and finitely many exposed sites, then for each $\tau$ there is a coupling of $\eta_t$ and $\zeta_t$ so that with probability tending to $1$ as $\tau\rightarrow\infty$,
\begin{itemize}
\item $\zeta_t=\eta_t$ for $t \in [0,T]\setminus S$ and 
\item $\ell(S\cap [0,T])\rightarrow 0$ where $\ell$ is Lebesgue measure on the line.
\end{itemize}
\end{theorem}
The main idea of the proof is that with probability tending to $1$ as $\tau\rightarrow\infty$ in the SEIS process, between any two onset events a recovery event occurs, and when this happens the SEIS process behaves like the limit process.  The assumption of finitely many initially active sites is necessary.\\

Unfortunately, unlike the contact process or the limit process, with respect to the graphical representation given above, for $\tau>0$ the SEIS process is not monotone in the partial order induced by the order $0<1<2$ on types (or, it can be checked, for any other order, though $0<2<1$ is the only other real possibility), since if we take configurations $\eta\leq\eta'$ with $\eta(x)=1$ and $\eta'(x)=2$ the $2$ can flip to a $0$ before the $1$ becomes a $2$, since type $1$ ignores the $\times$ labels.  Intuitively, this makes sense because although type $2$ can spread the infection while type $1$ cannot, type $1$ is not vulnerable to recovery events while type $2$ is.  Of course, it is possible to search for other graphical representations, or even more general types of coupling, to try to show monotonicity.  After a long search, we have found no such coupling, but the reader is encouraged to try!\\

So, lacking monotonicity, we define the following two critical values for the SEIS process; note $\mathbb{P}_{\lambda,\tau}$ denotes the law of the process with parameters $\lambda,\tau$.
\begin{eqnarray*}
\lambda_c^-(\tau) &=& \sup\{\lambda': \mathbb{P}_{\lambda,\tau}(\eta_t \textrm{ dies out }| |\eta_0|<\infty)=1 \textrm{ if }\lambda<\lambda'\}\\
\lambda_c^+(\tau) &=& \inf\{\lambda': \mathbb{P}_{\lambda,\tau}(\eta_t \textrm{ survives }| |\eta_0|>0)>0 \textrm{ if } \lambda>\lambda'\}
\end{eqnarray*}
Clearly, $\lambda_c^-(\tau)\leq\lambda_c^+(\tau)$ for each $\tau$.  The next result gives quantitative estimates on critical values, both for the SEIS process and for the limit process, on $\mathbb{Z}$, i.e., on the graph $G=(V,E)$ with $V=\mathbb{Z}$ and $E= \{xy:|x-y|=1\}$.\\

\begin{theorem}\label{thmquant}
For the SEIS process on $\mathbb{Z}$, $\lambda_c^+(\tau)<6.875$ when $\tau\leq 1/10$, and $\lambda_c^-(\tau)$ has the lower bounds given in Table 1.  For the limit process on $\mathbb{Z}$, $1.944 < \lambda_c^{\infty} < 8.563$.
\end{theorem}

\begin{table}
\caption{Lower bounds on $\lambda_c^-(\tau)$}
\centering{\begin{tabular}{r | r r r r r r r r r r } $\tau$ & $10^4$ & $10^3$ & $100$ & $10$ & $1$ & $0.58$ & $1/10$ & $1/100$ & $10^{-3}$ & $10^{-4}$ \\ \hline $\lambda_c^-(\tau)>$  & $1.57$ & $1.57$ & $1.56$ & $1.45$ & $1.15$ & $1.13$ & $1.24$ & $1.32$ & $1.34$ & $1.34$ \\
\end{tabular}}
\end{table}

\hspace{1in}

Lower bounds on $\lambda^-$ are obtained using the method of \cite{ziezold} applied to a monotone process that upperbounds the SEIS process, and the upper bound on $\lambda^+$ for small $\tau$ is obtained with the method of \cite{block} applied to a monotone process that lowerbounds the SEIS process.  In both cases the estimates are achieved with the assistance of a computer and are rigorous up to the rounding error on computations.  Unfortunately, in this case each lower bound on $\lambda^-$ is computed for a single value of $\tau$; it is possible to make guesses by interpolating, but these are not a priori rigorous.  Note also that the lower bounds suggest, but again do not prove, that the critical value of the upperbound process has a unique minimum near $\tau=0.58$ and is otherwise increasing/decreasing.  Numerical simulations of the SEIS process on $\mathbb{Z}$ suggest that $\lambda_c^-(\tau) = \lambda_c^+(\tau)$ and that this value increases monotonically from about $1.6$ at $\tau=0$ to about $2.4$ as $\tau\rightarrow\infty$.\\

For the limit process, the lower bound is obtained using the method of \cite{ziezold} and the upper bound, using the method of \cite{block}.  Note that for the contact process, $1.539 \leq \lambda_c \leq 1.942$ (lower bound from \cite{ziezold} and upper bound from \cite{impr}), and from the upper bound together with our estimate we note that the strict inequality $\lambda_c^{\infty}>\lambda_c^0$ holds.\\

Using different methods, we obtain some ``qualitative'' estimates on critical values.
\begin{theorem}\label{thmqual}
For the SEIS process on $\mathbb{Z}$,
\begin{itemize}
\item there exists $\lambda_0<\infty$ such that $\lambda_c^+(\tau)<\lambda_0$ for all $\tau$,
\item $\lambda^+(\tau),\lambda^-(\tau) \rightarrow \lambda_c^0$ as $\tau\rightarrow 0$ and
\item $\lambda^+(\tau),\lambda^-(\tau) \rightarrow \lambda_c^{\infty}$ as $\tau\rightarrow\infty$.
\end{itemize}
\end{theorem}
Here we show only that $\lambda_0<\infty$ exists, as it appears difficult to get any sort of realistic estimate.  The proof uses the block construction idea of \cite{block} with a bit of extra work to get around the lack of monotonicity.  Convergence of $\lambda^+,\lambda^-$ as $\tau\rightarrow 0$ is proved with the help of the results of \cite{supercrit} and \cite{ziezold}, in both cases by passing to a sequence of finite systems and using continuity with respect to parameters.  Convergence of $\lambda^+,\lambda^-$ as $\tau\rightarrow \infty$ is proved in the same way, with a couple of technical points that first need to be proved for the limit process.\\  

The paper is laid out as follows.  In Section \ref{seclimit} we prove Theorem \ref{thmlimit}.  In Section \ref{secquant} we prove Theorem \ref{thmquant}.  In Section \ref{secqual} we prove Theorem \ref{thmqual}.

\section{Theorem \ref{thmlimit}: Convergence to the Limit Process}\label{seclimit}
Here we prove Theorem \ref{thmlimit}.  We begin with a useful lemma.  Using the graphical representation given in Section \ref{secintro}, construct the SEIS process $\eta_t$ rescaled by $t\mapsto t/\tau$, so that onset occurs at rate $1$, recovery at rate $\tau$ and transmission at rate $\lambda\tau$.  Recall that $\mathcal{S} = G\times[0,\infty)$ denotes the spacetime set.
\begin{lemma}\label{onset}
Let $G=(V,E)$ be a finite graph.  In the rescaled SEIS process, for each $T>0$ with probability tending to $1$ as $\tau\rightarrow\infty$, for each onset label $\star$ at a point $(x,t) \in \mathcal{S}$ there is a $t'>t$ and a recovery label $\times$ at $(x,t')$ such that there are no onset labels in $V\times (t,t']$.
\end{lemma}
\begin{proof}
Let $\{(x_i,t_i):i=1,2,...\}$ be the set of points $(x,t) \in \mathcal{S}$ such that there is a $\star$ label at $(x,t)$, with $t_1<t_2<...$; since the total intensity of $\star$ labels is finite, with probability 1 the times can be ordered in this way.  Say a \emph{discrepancy} occurs at time $t_{i+1}$ if in the interval $\{x_i\}\times(t_i,t_{i+1})$ there are no $\times$ labels, then the desired event holds if the first discrepancy occurs after time $T$.  The intensity of $\times$ labels at each site is $\tau$ and the intensity of $\star$ labels is $|V|$, so for each $N$, with probability $[\tau/(|V|+\tau)]^N$ which $\uparrow 1$ as $\tau\rightarrow \infty$, there are no discrepancies up to time $t_{N+1}$.  Since $\mathbb{P}(t_{N+1}>T)\uparrow 1$ as $N\rightarrow\infty$, the result follows.
\end{proof}

\begin{proof}[Proof of Theorem \ref{thmlimit}]
We prove the result when $G$ is a finite graph.  The result for infinite graphs is implied by the following fact that can be seen from the construction of the process.  Fix $T>0$ and let $\eta_0$ be an initial condition for the SEIS process with no infectious site and finitely many exposed sites.  Define the graph distance $\dist(x,y)$ to be the least number of edges in any path between $x$ and $y$ with $\dist(x,y)=\infty$ if there is no path from $x$ to $y$.  For $k\geq 0$ let $G_k$ be the graph induced by the set of vertices $y\in V$ such that $\dist(x,y)<k$ for some $x$ such that $\eta_0(x)=1$ and let $\mathcal{S}_k = G_k \times [0,\infty)$, then define $\eta_t^k$ for $0 < t \leq T$ using the restriction of the substructure to $\mathcal{S}_k$.  Then, $\eta_t^k = \eta_t$ for $0 <t \leq T$ with probability tending to $1$ as $k\rightarrow\infty$.\\

First use the graphical representation to build two independent substructures $U^{(1)}_t$ and $U^{(2)}_t$ with respective filtrations $\mathcal{F}^{(1)}(t)$ and $\mathcal{F}^{(2)}(t)$, one as for the rescaled SEIS process and another as for the limit process.  Construct the rescaled SEIS process $\eta_t$ from $\eta_0$ using $U^{(1)}_t$ and let $\{(x_i,t_i);i=1,...,m\}$ with $t_1<t_2<...<t_m$ denote the points $(x,t)$ at which $\eta_t(x)$ goes from $1$ to $2$.  For $i=1,...,m$ let $s_i = \min\{t>t_i:\eta_t(x)=0\}$ be the first recovery time of $x_i$ after $t_i$.  Say that a \emph{discrepancy} occurs if $t_{i+1}<s_i$ for some $i \in \{1,...,m-1\}$.  So long as no discrepancy has occurred we will use $U^{(1)}_t$ to construct the limit process $\zeta_t$; $U^{(2)}_t$ will help to construct $\zeta_t$ in the event of a discrepancy.\\

For $i=1,...,m-1$ define the stopping times $r_i = \max s_i,t_{i+1}$, and let $\zeta_t=\eta_t$ for $t \in [0,t_1]$.  For $t \in [0,t_1]$ let $\zeta_t=\eta_t$.  Then, working inductively, suppose $\eta_t$ is determined for $t \in [0,r_{j-1}]$, is measurable with respect to $\mathcal{F}^{(1)}(r_{j-1})$, and no discrepancy has occurred up to time $r_{j-1}$ i.e., $r_i=t_{i+1}$ for $i=1,...,j-1$.  To determine $\zeta_t$ for $t \in [r_{j-1},t_{j+1}]$, use the $\leftrightarrow$ labels in $\{x_j\cdot\}\times (t_j,s_j)$ and the $\times$ label at $(x_j,s_j)$ to obtain the propagation distribution at $(x_j,t_j)$ and use the $\star$ label at $(x_{j+1},t_{j+1})$ to obtain the next onset transition; note these transitions depend only on $\zeta_{r_{j-1}}$ and $\mathbb{E}(U_{t+r_{j-1}}|\mathcal{F}(r_{j-1}))$ so have the correct distribution, and are measurable with respect to $\mathcal{F}(r_j)$.  If $r_j=t_{j+1}$ then $\zeta_t$ is determined for $t \in [0,r_j]$ and no discrepancy has occurred up to time $r_j$.  If $r_j=s_j$ a discrepancy has occurred; in this case, use the second substructure to determine $\zeta_t$ for $t \in [t_{j+1},T]$, noting that the second substructure is independent of the first.  Proceeding in this way determines $\zeta_t$ for $t \in [0,T]$, and by Lemma \ref{onset}, with probability tending to $1$ as $\tau\rightarrow\infty$, there are no discrepancies in the time interval $[0,T]$ and so $\eta_t=\zeta_t$, $0 \leq t \leq T$.  It is left to the reader to show that for $S = \{t:\eta_t(x)=2 \textrm{ for some }x\}$, $\ell(S\cap [0,T])\rightarrow 0$ with probability tending to $1$ as $\tau\rightarrow\infty$; to do so it suffices to prove a slight refinement of Lemma \ref{onset}.
\end{proof}

\section{Theorem \ref{thmquant}: Quantitative Estimates}\label{secquant}
In this section we prove Theorem \ref{thmquant}, in three parts:  estimate of $\lambda^+$, estimates of $\lambda^-$, estimate of $\lambda_c^{\infty}$.  First we define the lowerbound and upperbound processes, which we denote $\underline{\eta_t}$ and $\overline{\eta_t}$.  The idea is to modify some transitions in the SEIS process so that we end up with a monotone process that either lowerbounds or upperbounds the original process.  The definition is only relevant for $\tau>0$, since if $\tau=0$, in both cases it coincides with the contact process.\\

\subsection{Lowerbound Process}\label{seclb}
Starting from the graphical representation for the SEIS process, to obtain $\underline{\eta_t}$, construct the process as if it was the SEIS process, except that whenever an exposed site sees a recovery label $\times$, it becomes healthy.  As it turns out, this gives a particular case of what is called the two-stage contact process \cite{krone}, \cite{fox}, which is known to be monotone, as well as monotone increasing in $\lambda$ and monotone decreasing in $\tau$, with respect to the partial order on configurations induced by the order $0<1<2$ on types; given what was shown for the contact process in Section \ref{secintro}, this is not hard to check.  Intuitively, the reason why the stated monotonicity holds is because now the exposed type is in every sense weaker, in its ability to spread the infection, than the infectious type.  Monotonicity in $\lambda$ is intuitively clear, and monotonicity in $\tau$ can be explained by saying that the longer an exposed site has to wait to become infectious, the less it will spread the infection.  This gives the existence of a critical value $\underline{\lambda_c}(\tau)$ that is non-decreasing in $\tau$.  It is also not hard to check that $\underline{\eta_0}$ is a genuine lower bound, that is, if $\underline{\eta_0}\leq \eta_0$ then $\underline{\eta_t}\leq \eta_t$ for $t>0$, with respect to the partial order just described.  This implies in particular that $\underline{\lambda_c}(\tau)\geq\lambda^+(\tau)$ for each $\tau$.  As shown in \cite{fox}, for a graph of bounded degree, $\underline{\lambda_c}(\tau)\rightarrow \infty$ at a finite value of $\tau$, so this upper bound is only useful for small values of $\tau$.\\

\subsection{Upperbound Proces}\label{secub}
To get $\overline{\eta_t}$ we first picture $\eta_t$ as follows.  Recall that $\mathcal{S} = G \times [0,\infty)$ is the spacetime set, which we picture as a copy of $G$ extruded upward in the increasing time direction.  Given $\eta_0$, and determining the process on each realization for all time, if $\eta_t(x)=1$ for $t \in [t_1,t_2)$, draw a thick dashed line on the line segment $\{x\}\times[t_1,t_2)$ in $\mathcal{S}$, and if $\eta_t(x)=2$ draw a thick solid line; if $\eta_t(x)=0$ leave it blank.  If the infection is transmitted along an edge $e \in E$ then draw a thick solid line with an arrow pointing in the direction it was transmitted.\\

Now, modify the graphical representation so that transmission labels are directed.  That is, for $xy \in E$, with intensity $\lambda$ place transmission labels $\rightarrow$ from $x$ to $y$ and with intensity $\lambda$, place independent transmission labels $\leftarrow$ from $y$ to $x$.  Clearly, this does not change the law of $\eta_t$.  Then, to define $\overline{\eta_t}$ we simply allow both a dashed line \emph{and} a solid line to exist at the same site, at the same time; that is, if both $x$ and $y$ are infectious and there is a transmission label from $x$ to $y$, then $y$ becomes ``both'' infectious and exposed; with respect to the above visualization, along $y$ there is both a dashed line and a solid line, each behaving as it would in the SEIS process.  This is why we make labels directed; if both $x$ and $y$ have a solid line and a $\leftrightarrow$ appears along edge $xy$ there is no way to tell which of the two sites $x,y$ will receive a dashed line.\\

So, if $y$ has both a dashed and solid line and the next event is
\begin{itemize}
\item an onset label $\star$, the dashed line coalesces with the solid line, and only a solid line remains, and if it is
\item a recovery label $\times$, the solid line is knocked out and only the dashed line persists
\end{itemize}
Then, as in the SEIS process, only a solid line is able to use the transmission labels.  Also, at most one line of each type is allowed at a single site, so if $x$ already has a dashed line and there is a transmission event to $x$, it still has only one dashed line.  In order to refer to it, we denote by type 3 the presence of both a dashed and solid line.  It is not hard to check that $\overline{\eta_t}$ is monotone, and is monotone increasing in $\lambda$, with respect to the partial order on configurations induced by the order $0<1,2<3$ on types; it is not, however, monotone in $\tau$, effectively because in this process, as in the SEIS process, types $1$ and $2$ are not comparable.  Thus for each $\tau$, $\overline{\eta_t}$ has a critical value $\overline{\lambda}_c(\tau)$ whose variation in $\tau$ is not known a priori.  Clearly, $\overline{\eta_t}$ is an upper bound for $\eta_t$ in the sense that $\overline{\eta_0}\geq \eta_0$ implies $\overline{\eta_t}\geq\eta_t$ for $t>0$, with respect to the partial order just described, which implies that $\overline{\lambda}_c(\tau) \leq \lambda^-(\tau)$, for each $\tau$.\\

\subsection{Some Definitions}
We introduce a couple of definitions that will be useful in this section and the next section.  Notice that, given a finite state space $S$ and transition rates $q_{ij}$ for $i,j \in S$, we can construct a continuous time Markov chain on $S$ as follows.  Given that $X_t=i$, then corresponding to the set of $k$ such that $q_{ik}>0$ we have an independent collection of exponential random variables $s_k$ with rate $q_{ik}$, and setting $s=\min s_k$ and $j = \argmin s_k$, we let $X_{t+s} = j$.  Note this is the same Markov chain as the one in which each $i$ to $j$ transition is generated by a Poisson point process with intensity $q_{ij}$.  It can be checked by a calculation that $p_{ij} = \mathbb{P}(X_{t+s}=j|X_t=i) = q_{ij}/q_i$ if $q_i\neq 0$, where $q_i = \sum_{k \neq i}q_{ik}$; if $q_i=0$ the chain remains at $i$ once it has arrived there.\\
\begin{definition}
The discrete time Markov chain on $S$ determined by the transition probabilities $p_{ij}$ defined above is called the \emph{embedded jump chain}.
\end{definition}
Note that the embedded jump chain encodes the changes of state in the continuous time chain.  In fact, the continuous time chain can be reconstructed from the embedded jump chain by waiting an independent exponential time of rate $q_i$ at each state $i$ before making a transition; see \cite{norris} for details.\\

Our next task is to define a notion of path for the infection in spacetime.
\begin{definition}\label{defpath}
A \emph{path} in spacetime is a list of alternating vertical and horizontal line segments
\begin{equation*}
(v_1,h_1,...,v_{m-1},h_{m-1},v_m)
\end{equation*}
in $\mathcal{S}$ with each $v_i = \{x_i\}\times(t_{i-1},t_i)$, $t_{i-1}<t_i$ and each $h_i = \{x_ix_{i+1}\}\times \{t_i\}$.  The \emph{base} of a path is the point $(x_1,t_0)$, and the \emph{end} is the point $(x_m,t_m)$.\\
\end{definition}
The notion of active path is here defined only for the contact, lowerbound, and limit processes.  Although it is easy to define for the upperbound process we will have no need for it, and for the SEIS process the definition will be a bit different.
\begin{definition}\label{defactive}
For the contact process, a path is \emph{active} if for $i=1,...,m-1$ there is a $\leftrightarrow$ label at $h_i$ and for $i=1,..,m$ there are no $\times$ labels on $v_i$.  For the lowerbound process, a path is active if in addition, for $i=1,...,m-1$ there is a $\star$ label on $v_i$.  For the limit process, a path is active if instead of a $\leftrightarrow$ label at $h_i$, $x_{i+1}$ belongs to the offspring of $(x_i,t_i)$.
\end{definition}

Say a point $(x,t)$ is \emph{active} for a process $\eta_t$ if $\eta_t(x)\neq 0$.  It is easy to check for the contact process, the lowerbound process and the limit process that if $(x,t)$ is active and there is an active path with base $(x,t)$ and end $(y,s)$ then $(y,s)$ is active; one way to express this is that active paths take active points to active points.  It is also true that $(y,s)$ is active if and only if there is an active point $(x,0)$ and an active path with base $(x,0)$ and end $(y,s)$.  This implies a useful property called \emph{additivity} which is discussed in the proof of Lemma \ref{limcomp} in Section \ref{secconcat}.

\subsection{Estimate of $\lambda^+$}
We will use the method described in \cite{block}, applied to the lowerbound process on $\mathbb{Z}$, to obtain upper bounds on $\underline{\lambda_c}(\tau)$.  Our first task is to describe the method and to justify its usage in this setting.  Recall that $\mathcal{S}=G\times[0,\infty)$ is the spacetime set, for $R \subset \mathcal{S}$, $\mathcal{F}(R)$ is the $\sigma$-algebra generated by the restriction of the substructure of $R$.  In what follows, we say an event holds ``on $R$'' if it is $\mathcal{F}(R)$-measurable.\\

Letting $L:=\{(m,n)\in\mathbb{Z}^2:n\geq 0, m+n\textrm{ is even}\}$, define \emph{oriented site percolation} on $L$ to be the process in which sites in $L$ are independently \emph{open} with probability $p$ and \emph{closed} with probability $1-p$, and there is a \emph{path} from $(k,l)$ to $(m,n)$ or $(k,l)\rightarrow(m,n)$ if there is a list $(k,l)=(k_1,l_1),(k_2,l_2),...,(k_j,l_j)=(m,n)$ such that $l_{i+1}=l_i+1$ and $k_{i+1}=k_i\pm 1$ for $i=1,...,j-1$, and $(k_i,l_i)$ is open for $i=1,...,j-1$; note the last site is not required to be open.  The \emph{cluster} of $(m,n)$ is the set $C(m,n) = \{(k,l) \in L:(m,n)\rightarrow (k,l)\}$.  We say that \emph{percolation occurs} from $(m,n)$ if $|C(m,n)|=\infty$.  If $p<1$ is close enough to $1$, then percolation occurs from $(0,0)$ with positive probability; as shown in \cite{block}, $p\geq 0.819$ suffices.\\

The idea of the method of \cite{block} is to use the above result for oriented percolation to prove survival of the infection in the process of interest.  In our case the discussion applies to the lowerbound process, the contact process, the limit process, and any process whose active paths take active points to active points as described above.  Embed the spacetime set $(\mathbb{Z},\{xy:|x-y|=1\})\times[0,\infty)$ into the half-space $\{(x,y)\in\mathbb{R}^2:y\geq 0\}$ then draw the rectangles $\{R_{m,n}:(m,n) \in L\}$ defined by
\begin{equation*}
R_{m,n} = R_{0,0} + (mK,nT)
\end{equation*}
for some $K,T$ to be determined, where $R_{0,0} = [0,J]\times[0,T]$ for some integer $K\leq J<2K$ and for a set $S$ and a point $r$, $S+r:=\{s+r:s \in S\}$; the range of $J$ ensures that $R_{m,n}$ does not intersect $R_{m-2,n}$ or $R_{m+2,n}$, but does intersect $R_{m-1,n+1}$ and $R_{m+1,n+1}$ along its top edge; let $j$ denote the number of sites along which $R_{m,n}$ intersects $R_{m-1,n+1}$, which is also the number of sites along which $R_{m,n}$ intersects $R_{m+1,n+1}$.\\

Fix a parameter $i \in \{1,...,j\}$.  Say that a configuration $\eta$ is good for $R_{m,n}$ if among either the leftmost or rightmost $j$ sites in $R_{m,n}$ there are at least $i$ distinct sites $x_1,...,x_i$ such that $\eta(x_k)\neq 0$ for $k=1,...,i$.  For $(m,n) \in L$, and for $\eta$ that is good for $R_{m,n}$, define $A_{m,n}(\eta)$ as follows:  $A_{m,n}(\eta)$ occurs if for at least $i$ of the $j$ leftmost sites $y_1,...,y_i$ and at least $i$ of the $j$ rightmost sites $y_{i+1},...,y_{2i}$ in $R_{m,n}$ at time $(n+1)T$, for $p=1,...,2i$ there is an $q \in \{1,...,i\}$ such that there is an active path lying in $R_{m,n}$ with base $(x_q,nT)$ and end $(y_p,(n+1)T)$.  See Figure \ref{block-fig} for a picture.\\

\begin{figure}
\centering{\includegraphics[width=130mm,height=80mm]{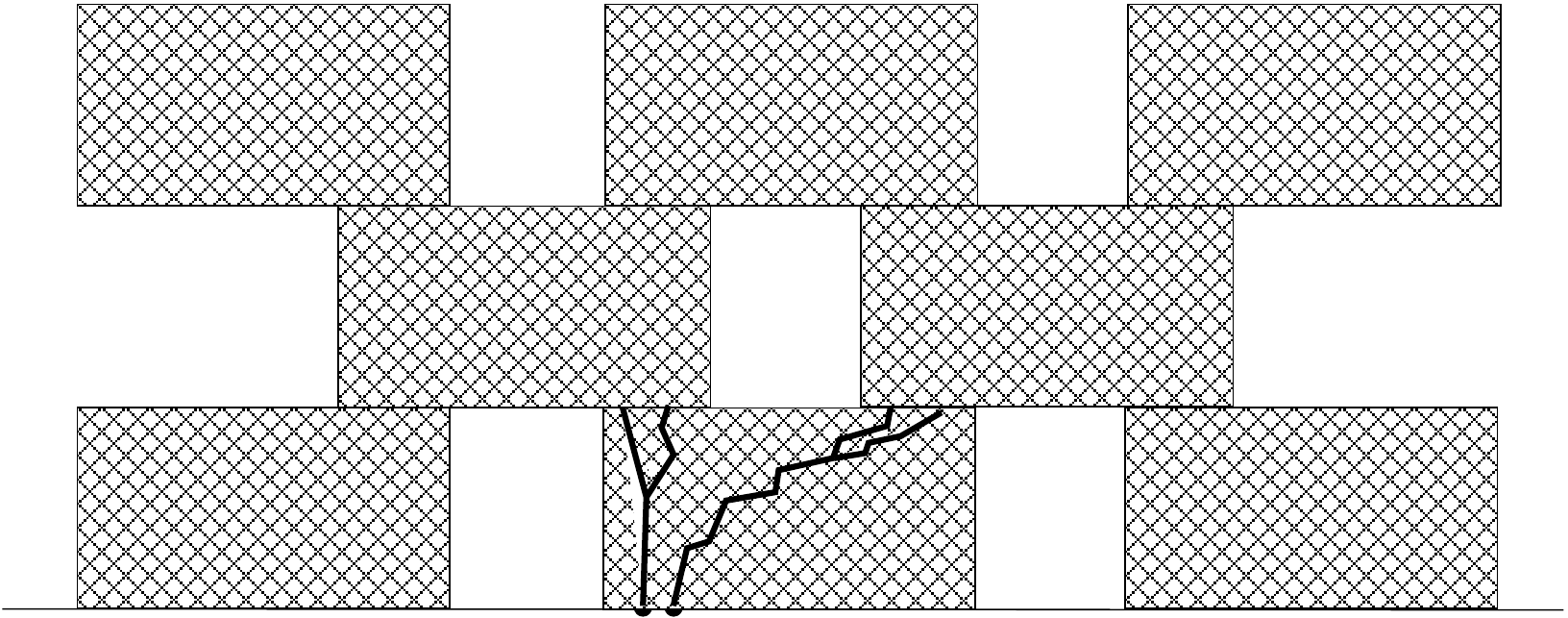}}
\caption{Depiction of rectangles $R_{m,n}$ for $-2 \leq m \leq 2$ and $0 \leq n \leq 2$, as well as the event $A_{0,0}(\eta_0)$ in $R_{0,0}$, with $i=2$.}
\label{block-fig}
\end{figure}

If $\eta_{nT}$ is good for $R_{m,n}$ and $A_{m,n}(\eta_{nT})$ occurs, then $\eta_{(n+1)T}$ is good for both $R_{m-1,n+1}$ and $R_{m+1,n+1}$.  Moreover, given $\eta_{nT}$ that is good for $R_{m_1,n},R_{m_2,n},...$, the events $A_{m_1,n}(\eta_{nT}),A_{m_2,n}(\eta_{nT}),...$ are independent.  It is then straightforward to show that if for each $(m,n) \in L$, $\mathbb{P}(A_{m,n})\geq p$ for every configuration that is good for $R_{m,n}$ then if $\eta_0$ is good for $R_{0,0}$, the set $\{(m,n)\in L:\textrm{ there is an active point for }\eta_t \textrm{ in }R_{m,n}\}$ stochastically dominates the cluster $C(0,0)$ of oriented site percolation with parameter $p$.  Moreover, taking $\eta$ at time $0$ that is good for $R_{0,0}$ together with the set of labellings of $R_{0,0}$ belonging to the event $A_{0,0}$ given $\eta$ and translating both by $(mK,nT)$ gives $\eta'$ that is good for $R_{m,n}$ together with the set of labellings belonging to $A_{m,n}$ given $\eta'$.  Thus, by translation invariance of the law of the substructure, to show the lowerbound process survives with positive probability starting from $\eta_0$ that is good for $R_{0,0}$, it suffices to show that $\mathbb{P}(A_{0,0})\geq 0.819$ for each $\eta_0$ that is good for $R_{0,0}$.\\

Using blocks with $J=7$ and $K=6$ and setting $T$ so that there are an average of 650 labels in $R$ (650 just chosen to match the choice in \cite{block}), following \cite{block} we estimate numerically the transition matrix $P$ for the embedded jump chain corresponding to the lowerbound process restricted to the sites in $R$, counting redundant transitions (i.e. points in the Poisson process that have no effect) so that the rate of transitions is fixed and is equal to the total intensity of p.p.p.'s, that we denote $\gamma$, and then by computing the first couple of thousand terms in the sum
\begin{equation*}
\sum_{i=1}^{\infty}e^{-\gamma T}\frac{\gamma^i}{i!}P^i
\end{equation*}
by monotonicity we obtain a lower bound on the entries of the transition matrix at time $T$ in the continuous time chain.  If $\lambda=6.875$ and $\tau=1/10$, then letting $i=1$, with respect to the estimated transition matrix we find that for each configuration favourable for $R_{0,0}$, $A_{0,0}$ has probability at least $0.819$, which implies $\underline{\lambda_c}(1/10)\leq 6.875$ and by monotonicity in $\tau$, that $\lambda_c(\tau) < 6.875$ when $\tau \leq 1/10$, so that for the SEIS process $\lambda_c^+(\tau)< 6.875$ when $\tau \leq 1/10$.

\subsection{Estimates of $\lambda^-$}
For the upperbound process we use the method of \cite{ziezold}.  Starting with the upperbound process on $\mathbb{Z}$ with $\tau>0$ and initial configuration $\eta_0(x)=3$ for $x \geq 0$ and $\eta_0(x)=0$ for $x<0$, for integer $m\geq 0$ we define the modified process $\eta_t^m(x)$ by evolving like the upperbound process but with the added constraint $\eta_t^m(x)=3$ for $x > l_t^m+m$, where $l_t^m = \inf\{x:\eta_t^m(x)\neq 0\}$.  The vector $v(t) = (\eta_t^m(l_t^m),...,\eta_t^m(l_t^m+m)$ evolves like a finite state continuous time homogenous Markov chain on the state space $S = \{\eta \in \{0,1,2,3\}^{m+1}:\eta(0) \neq 0\}$, so we let $X_n$ denote the embedded jump chain, which evolves on the same space.  For this chain each state communicates with the state defined by $X(x)=3,x=0,...,m$, so the chain is irreducible on $S$ with a unique invariant measure that we denote $\mu$.  Letting $k(i,j)$ denote the number of distinct transitions in $\eta_t^m$ that take $v(t)$ from $i$ to $j$, and letting $\Delta(i,j,r)$ be the increment in $l_t^m$ and $p(i,j,r)$ be the transition probability for the $r^{th}$ transition, $1 \leq r \leq k(i,j)$, we define $\lambda_m$ as $\sup\{\lambda:\mathbb{E}_{\mu}\Delta > 0\}$, where
\begin{equation}\label{increq}
\mathbb{E}_{\mu}\Delta = \sum_{i,j \in S}\mu(i)\sum_{r=1}^{k(i,j)}p(i,j,r)\Delta(i,j,r)
\end{equation}
is the average increment in $l_t^m$ at each transition.  Since $\eta_t^m \geq \eta_t$ and $\mathbb{E}_{\mu}\Delta>0$ implies $l_t^m\rightarrow \infty$, if $\lambda<\lambda_m$ then $l_t\rightarrow\infty$.  A simple coupling argument as in \cite{speed} then shows that the upperbound process on $\mathbb{Z}$ started from a finite number of active sites dies out with probability 1 when $l_t\rightarrow\infty$, which implies $\lambda_m \leq \overline{\lambda}_c$.  To estimate $\lambda_m$ we construct the embedded jump chain, compute $\mu$, and iterate to find $\lambda$ such that $\mathbb{E}_{\mu}\Delta \approx 0$.  With $m=3$ we obtain the table of values given in Theorem \ref{thmquant}.

\subsection{Estimate of $\lambda_c^{\infty}$}
Defining an active path for the limit process as in Definition \ref{defactive}, we have that active paths take active points to active points, so we can use the method of \cite{block} as described.  Doing so with $L=10,j=4$ and $i=2$ and $T$ chosen to give an average of 650 labels gives the upper bound on $\lambda_c^{\infty}$.\\

To get a lower bound we use again the method of \cite{ziezold}.  From the limit process $\zeta_t$, for $m\geq 0$ the modified process $\zeta_t^m$ is defined by evolving like the limit process but with the added constraint $\zeta_t^m(x)=1$ for $x>l_t^m+m$, with $l_t^m$ as defined above.  The definition of $\lambda_m$, $\mu$ and $\mathbb{E}_{\mu}\Delta$ are as above, with $\mu$ now supported on the state space $\{\eta \in \{0,1\}^{m+1}:\eta(0)\neq 0\}$.  We give the coupling argument that shows $l_t\rightarrow\infty$ implies the limit process started from a finite number of active sites dies out with probability 1:  letting $\mathbf{1}$ denote the indicator function define $\zeta_t^-,\zeta_t^0$ and $\zeta_t^+$ by $\zeta_0^-= \mathbf{1}(x\leq 0)$, $\zeta_0^0 = \mathbf{1}(x=0)$ and $\zeta_0^+ = \mathbf{1}(x\geq 0)$, then by monotonicity $\zeta_t^0 \leq \min \zeta_t^-,\zeta_t^+$.  If $l_t^+ = \inf\{x:\zeta_t^+(x)\neq 0\} \rightarrow\infty$ then by symmetry $r_t^- = \sup\{x:\zeta_t^-(x)\neq 0\} \rightarrow -\infty$, moreover $r_t^-<l_t^+$ implies $\zeta_t^0 \equiv 0$, so there is almost surely finite time $s$ so that $\zeta_t^0 \equiv 0$ for $t>s$.  An analogous argument works when $\zeta_t^0$ is replaced with $\zeta_t^N$ defined by $\zeta_0^N = \mathbf{1}(x \in [-N,N])$.  Computing $\lambda_m$ for $m=8$ in the same way as above gives the lower bound on $\lambda_c^{\infty}$.

\section{Theorem \ref{thmqual}: Qualitative Estimates}\label{secqual}
In this section we prove Theorem \ref{thmqual}, in five parts: existence of $\lambda_0$, upper bound on $\lambda^+$ as $\tau\rightarrow 0$, lower bound on $\lambda^-$ as $\tau\rightarrow 0$, upper bound on $\lambda^+$ as $\tau\rightarrow \infty$, lower bound on $\lambda^-$ as $\tau\rightarrow \infty$.

\subsection{Existence of $\lambda_0$}
To show the existence of $\lambda_0$ we use a comparison to oriented site percolation, in the spirit of \cite{block}.  As in the previous section, let $L := \{(m,n) \in \mathbb{Z}^2: n \geq 0, m+n \textrm{ is even}\}$ and for $T$ to be determined, define the set of rectangles $\{R_{m,n}:(m,n) \in L\}$ by $R_{0,0} = R = [0,3]\times [0,T]$ and
\begin{equation*}
R_{m,n} = R + (2m,nT)
\end{equation*}
where for a set $S$ and a point $r$, $S+r:=\{s+r:s \in S\}$.  For the SEIS process on $\mathbb{Z}$, the graphical representation embeds in a natural way into the set $\{(x,y)\in\mathbb{R}^2:y\geq 0\}$, as do the rectangles $R_{m,n}$.  In light of the upper bound on $\lambda^+(\tau)$ given for $\tau\leq 1/10$ in Theorem \ref{thmquant}, it is enough to show that survival occurs with positive probability when $\lambda>\lambda_0$ and $\tau \geq \tau_0>0$ for some $\tau_0 \leq 1/10$; we phrase it in this way because we will be able to take $\tau_0>0$ as small as we choose.\\

For each $m$, let $x_m$ denote the site such that $R_{m,n} = [x_m,x_m+3]\times[nT,(n+1)T]$, when $m+n$ is even.  To show survival we define, for each $p<1$, a value $\lambda_0$ such that for each $\lambda$ and $\tau$ there is a time $T$ and a collection of events $A_{m,n}$ each one depending only on the corresponding $\mathcal{F}(R_{m,n})$, the restriction of the substructure to $R_{m,n}$, with the property that if $\eta_{nT}(y)\neq 0$ for some $y \in [x_m,x_m+3]$ and $A_{m,n}$ occurs, then $\eta_{(n+1)T}(y)\neq 0$ for some $y \in [x_m,x_m+1]$ \emph{and} some $y\in [x_m+2,x_m+3]$, and such that if $\lambda>\lambda_0$, then $\mathbb{P}(A_{m,n})\geq p$.  Then, provided $\eta_0(0)\neq 0$, if $\lambda>\lambda_0$ the set of $(m,n)$ such that $\eta_t(x)\neq 0$ for some $(x,t) \in R_{m,n}$ dominates the cluster $C(0,0)$ of oriented percolation with parameter $p$, so if $p$ is chosen close enough to $1$ we conclude that survival in the SEIS process occurs with positive probability uniformly in $\tau$, for any value $\lambda>\lambda_0$.\\

For simplicity we define $A_{m,n}$ to have the property that if $\eta_{nT}(y) \neq 0$ for some $y \in [x_m,x_m+1]$ then $\eta_{(n+1)T}(y) \neq 0$ for some $y \in [x_m,x_m+1]$ and some $y \in [x_m+2,x_m+3]$.  By reflection symmetry of the substructure, if $\mathbb{P}(A_{m,n}) \geq 1-\epsilon/2$ then the probability that the same conclusion holds with the weaker hypothesis $\eta_{nT}(y) \neq 0$ for some $y \in [x_m,x_m+3]$ is at least $1-\epsilon$.  By translation-invariance it is enough to define $A_{0,0}$, for which $x_m=0$.  Letting $B(t) = (\eta_t(0),\eta_t(1))$ and $C(t) = (\eta_t(2),\eta_t(3))$, say that either is \emph{active} at time $t$ if at least one of its two coordinates is not zero.  Then, it is enough that $A_{0,0}$ have the following two properties:
\begin{enumerate}
\item if $B(s)$ is active for some $s \in [0,T)$ then $B(t)$ is active for $s<t\leq T$, and the same is true for $C(t)$, and
\item if $B(0)$ is active (which it is by the assumption $\eta_0(0)\neq 0$) then for some $s \in [0,T)$, $C(s)$ is active
\end{enumerate}
Before choosing $T$ it is convenient to rescale the model in time so that $\star$ labels occur at rate $1$, $\times$ labels at rate $\tau$ and $\leftrightarrow$ labels at rate $\lambda\tau$; this is ok since $T$ is allowed to depend on $\tau$, and is convenient since after rescaling, it won't have to.\\

To get property 1 for both $B(t)$ and $C(t)$ (and a bit more besides) it is enough that the following event depending on a fixed parameter $h>0$, that we denote $E_1$, hold:  in the time interval $[0,h]$ there is a $\leftrightarrow$ label on each of the edges $01,12,23$ before there is a $\times$ label on any of the sites $\{0,1,2,3\}$ and everywhere on the set $[0,3]\times [0,T]$, after any $\star$ label on $\{0,1,2,3\}$ and within at most $h$ amount of time, there is a $\leftrightarrow$ on each of the edges $01,12,23$ before there is a $\times$ label on any of the sites $\{0,1,2,3\}$.  Since after rescaling, $\star$ labels occur at rate $1$, $\times$ labels at rate $\tau$ and $\leftrightarrow$ labels at rate $\lambda\tau$, it is a simple exercise to show that given $h,\tau_0>0$, for each $\epsilon,T>0$ we can choose $\lambda_0$ so that $\mathbb{P}(E_1)\geq 1-\epsilon$ if $\lambda>\lambda_0$ and $\tau>\tau_0$.\\

To get property $2$, it is enough to show that $t\leq T$ where $t$ is the first time such that $\eta_t(1)=2$ and a $\leftrightarrow$ on $12$ occurs at time $t$.  If $E_1$ holds, it is enough to have one of the following:
\begin{itemize}
\item $\eta_0(1)=2$ and $T\geq h$,
\item $\eta_0(1)=1$ and there is a $\star$ label at $(1,t)$ for some $t \leq T-h$,
\item $\eta_0(0)=2$ and there is a $\star$ label at $(1,t)$ for some $h<t \leq T-h$, and
\item $\eta_0(0)=1$ and there is a $\star$ label at $(0,t_1)$ for some $t_1 \leq T-2h$ and a $\star$ label at $(1,t_2)$ for some $t_2 \in (t_1+h,T-h]$
\end{itemize}
Since we can choose $T\geq h$ and since one of the four conditions on $B(0)$ holds by assumption, we denote by $E_2$ the intersection of the label events just described, so that $E_2$ depends only on $\mathcal{F}(R_{0,0})$ as desired.  Given $\epsilon>0$, it is not hard to check that for $T$ large enough, $\mathbb{P}(E_2) \geq 1-\epsilon/2$.  Then, we can choose $\lambda_0$ so that if $\tau>\tau_0$ and $\lambda>\lambda_0$, $\mathbb{P}(E_1)\geq 1-\epsilon/2$, so that $\mathbb{P}(E_1 \cap E_2)\geq 1-\epsilon$, and this completes the proof.

\subsection{Upper bound on $\lambda^+$ as $\tau\rightarrow 0$}\label{seccat}
We now show that $\lambda^+(\tau),\lambda^-(\tau) \rightarrow \lambda_c^0$, the critical value of the contact process, as $\tau\rightarrow 0$.  We first use the lowerbound process to show that $\limsup_{\tau\rightarrow 0}\lambda^+(\tau) \leq \lambda_c^0$.  In the proof we mention a $1$-dependent oriented site percolation process with parameter $p$; this is a model in which each site is open with probability $p$, and sites $(m_1,n_1),...,(m_k,n_k)$ are independent provided $|m_i-m_j|+|n_i-n_j|>2$ for $i \neq j$.  The definition of paths, clusters and percolation is the same as before.  As shown in \cite{supercrit}, for a $1$-dependent oriented site percolation process, if $p=1-\epsilon$ for $\epsilon>0$ small enough, then percolation from $(0,0)$ occurs with positive probability.\\

In \cite{supercrit} it is shown for the contact process that if $\lambda>\lambda_c^0$ then for each $\epsilon>0$, for a suitable choice of rectangles $R_{m,n}\in \mathcal{S}$ and $\mathcal{F}(R_{m,n})$-measurable events $A_{m,n}$ with indicator $I_{m,n}$, the set $\{(m,n)\in L:I_{m,n}=1\}$ dominates a 1-dependent oriented site percolation process with parameter $ p=1-\epsilon$.  The $1$-dependence arises from the fact that $R_{m,n}$ overlaps with each of $R_{m\pm 1,n\pm 1}$ on a set of positive measure in $\mathcal{S}$; see Fig. 1 in \cite{supercrit} for a picture.  The events $A_{m,n}$ are such that for appropriate choice of initial configuration with finitely many active sites, the set of $(m,n) \in L$ such that there is an active point in $R_{m,n}$ stochastically dominates $C(0,0)$, so to show survival it is \emph{sufficient} to have $\mathbb{P}(A_{m,n}) \geq 1-\epsilon$ for $\epsilon>0$ small enough.\\

Now, since each rectangle $R_{m,n}$ is bounded, if $\mathbb{P}_{\lambda}(A_{m,n})>p$ then for small enough $\delta>0$, $\mathbb{P}_{\lambda-\delta}(A_{m,n})>p$.  To see this, proceed as in Section \ref{secintro}:  first generate transmission labels $\leftrightarrow$ and $\leftrightarrow'$ using independent p.p.p's with intensity $\lambda-\delta$ and $\delta$ respectively, then let the process with transmission parameter $\lambda$ use both types of labels, and let the process with transmission parameter $\lambda-\delta$ use only the labels $\leftrightarrow$.  Since $R$ is a bounded region in spacetime, with probability tending to $1$ as $\delta\rightarrow 0$ there are no $\leftrightarrow'$ labels in $R$, and on this event the p.p.p.'s for the two processes agree on $R$.  Viewing the contact process as the lowerbound process with $\tau=0$, a similar argument shows that if $\mathbb{P}_{\lambda,0}(A_{m,n})>p$ then for small enough $\delta$ and $\tau$, $\mathbb{P}_{\lambda-\delta,\tau}(A_{m,n})>p$; for a detailed argument see \cite{fox}, but the main idea is that for a bounded region in spacetime, with probability tending to $1$ as $\tau\rightarrow 0$, between any two consecutive labels of type $\leftrightarrow$ or $\times$ in time there is a $\star$ label at every site.  Therefore if $\lambda>\lambda_c^0$ then taking $p=1-\epsilon$ where $\epsilon>0$ is such that percolation occurs with positive probability in a $1$-dependent oriented site percolation model with parameter $p$, for small enough $\delta,\tau>0$, started from some initial configuration with finitely many active sites the lowerbound process with parameters $\lambda-\delta,\tau$ survives with positive probability.  In other words, $\lambda>\lambda_c^0$ implies $\lambda>\lambda_c(\tau)$ for small enough $\tau>0$, and noting that $\lambda^+(\tau)\leq\lambda_c(\tau)$ the desired result follows.

\subsection{Lower bound on $\lambda^-$ as $\tau\rightarrow 0$}
Next we show that $\liminf_{\tau\rightarrow 0}\lambda^-(\tau) \geq \lambda_c^0$, using the upperbound process and a comparison.  In Section \ref{secquant} we described the method of \cite{ziezold} that gives, for each $\tau>0$, a sequence of lower bounds $\lambda_0(\tau)\leq\lambda_1(\tau)\leq\lambda_2(\tau)\leq...\leq\overline{\lambda}_c(\tau)$, each is which is determined by a process $\eta_t^m$ that approximates the upperbound process; we did not show that the $\lambda_m$ are increasing but this follows from the observation $\eta_t^m \geq \eta_t^{m+k}$ for $k\geq 0$.  For $\tau=0$ which is the contact process, starting from the process with initial configuration $\eta_0(x)=2$ for $x \geq 0$ and $\eta_0(x)=0$ for $x<0$, then for integer $m\geq 0$ define $\eta_t^m$ by evolving like the contact process but with the added constraint $\eta_t^m(x)=2$ for $x>l_t^m+m$, with again $l_t^m= \inf\{x:\eta_t^m(x)\neq 0\}$.  The definition of $\lambda_m$, $\mu$ and $\mathbb{E}_{\mu}\Delta$ are as before.  Since $\eta_t^m \geq \eta_t^{m+k}$ we have again $\lambda_0(0)\leq \lambda_1(0) \leq ... \lambda_c^0$.  It is shown in \cite{ziezold} that $\lambda_m(0)\uparrow\lambda_c^0$; the proof relies on the fact, proved as Theorem 4 in \cite{basic}, that $\lambda_c^0=\sup\{\lambda:\alpha > 0\}$, where for $l_t =\inf\{x:\eta_t(x)\neq 0\}$, $\alpha = \lim_{t\rightarrow\infty}l_t/t$ was shown to exist in \cite{speed}.\\

For $\tau>0$, $\mu$ is supported on $S = \{\eta \in \{0,1,2,3\}^{m+1}:\eta(0)\neq 0\}$ and for $\tau=0$, $\mu$ is supported on $S_0 = \{\eta \in \{0,2\}^{m+1}:\eta(0)\neq 0\} \subset S$.  Define $S_1 = \{\eta \in S:\eta(x)\in \{1,3\} \textrm{ for some }x \textrm{ and }\eta(y)\in \{0,2\}, y \neq x\}$, the set of states with exactly one exposed site.  Modify the embedded jump chain $X_n$ when $\tau=0$ to include intermediate transitions to configurations with an exposed i.e. type 1 or type 3 site; $\mu$ is then supported on $S_0\cup S_1$ and from any state in $S_1$ there is a unique transition with probability 1 to the corresponding state in $S_0$ in which onset of the exposed site has occurred.  Since this modification preserves the sign of $\mathbb{E}_{\mu}\Delta$ the value of $\lambda_m$ is unchanged.  Writing $\mathbb{E}_{\mu}\Delta(m,\lambda,\tau)$ to emphasize the dependence and noting that $\lambda_m(\tau) = \sup\{\lambda:\mathbb{E}_{\mu}\Delta(m,\lambda,\tau)>0\}$, to prove the result it suffices to show that for each $m$ and $\lambda$, $\mathbb{E}_{\mu}\Delta(m,\lambda,\tau)$ is continuous in $\tau$ at $\tau=0$, since if $\lambda<\lambda_c^0$ then $\lambda<\lambda_m(0)$ for some $m$ and $\mathbb{E}_{\mu}\Delta(m,\lambda,0)>0$ and then by continuity $\mathbb{E}_{\mu}\Delta(m,\lambda,\tau)>0$ and thus $\lambda<\lambda_m(\tau)$ for small enough $\tau>0$ and the result follows from the inequalities $\lambda_m(\tau)\leq \overline{\lambda}_c(\tau) \leq \lambda^-(\tau)$.\\

Recall equation \eqref{increq}:
\begin{equation*}
\mathbb{E}_{\mu}\Delta = \sum_{i,j \in S}\mu(i)\sum_{r=1}^{k(i,j)}p(i,j,r)\Delta(i,j,r)
\end{equation*}
Thus to show $\mathbb{E}_{\mu}\Delta(m,\lambda,\tau)\rightarrow \mathbb{E}\Delta(m,\lambda,0)$ as $\tau\rightarrow 0$ it suffices to show $\mu(i)(m,\lambda,\tau) \rightarrow \mu(i)(m,\lambda,0)$ for $i \in S$ and $p(i,j,r)(m,\lambda,\tau)\rightarrow p(i,j,r)(m,\lambda,0)$ as $\tau\rightarrow 0$ for $i,j \in S$ and $r=1,...,k(i,j)$.  In fact, we can make a further reduction.\\

\begin{lemma}
Let $p(i,j)(s)$, $0 \leq s \leq 1$, be a family of transition probabilities on a finite state space $S$ such that for each $s$, $p(i,j)(s)$ has a unique invariant measure $\mu(i)(s)$.  If $p(i,j)(s) \rightarrow p(i,j)(0)$ as $s \rightarrow 0$ for each $i,j \in S$ then $\mu(i)(s) \rightarrow \mu(i)(0)$ as $s\rightarrow 0$ for $i \in S$.
\end{lemma}
\begin{proof}
Let $n = |S|$ and define the simplex $\Lambda = \{x \in \mathbb{R}^n:x_i\geq 0, i=1,...,n,\,\,\sum_i x_i =1\}$ that corresponds to probability measures on $S$.  Suppose by way of contradiction that there is a sequence $(s_k)$ tending to $0$ and an $\epsilon>0$ such that $\max_i |\mu(i)(s_k)-\mu(i)(0)|>\epsilon$ for each $k$.  By compactness of $\Lambda$ there is a subsequence $(s_{k_m})$ tending to $0$ and an element $\mu^* \in \Lambda$ with $\mu^*\neq \mu(0)$ such that $\mu(i)(s_{k_m}) \rightarrow \mu^*(i)$ for each $i$.  However, since $p(i,j)(s_{k_m})\rightarrow p(i,j)(0)$ for each $(i,j)$ and each $\mu(s_{k_m})$ is invariant for $p(i,j)(s_{k_m})$, $\mu^*$ is invariant for $p(i,j)(0)$, and by assumption of uniqueness, $\mu^* = \mu(0)$, a contradiction.
\end{proof}

By the lemma above, since for $X_n$, $p(i,j) = \sum_{r=1}^{k(i,j)}p(i,j,r)$, and $\mu$ is determined from $p(i,j)$ it suffices to show $p(i,j,r)(m,\lambda,\tau)\rightarrow p(i,j,r)(m,\lambda,0)$ as $\tau\rightarrow 0$.  For each $i,j \in S$ and $r \in \{1,...,k(i,j)\}$ the transition rate $q_{ijr}$ for $\eta_t^m$ is of the form $1/\tau$, $\lambda$, $2\lambda$ or $1$.  Using the formula $p_{ijr} = q_{ijr}/\sum_{(k,s) \neq (i,r)}q_{iks}$ and writing $p_{ijr}(m,\lambda,\tau)$ to emphasize the dependence, fix $m$ and $\lambda$.  For $i \in S_0$, $p_{ijr}(m,\lambda,\tau)$ does not depend on $\tau$ and agrees with $p(i,j,r)(m,\lambda,0)$ provided we include intermediate transitions in $X_n$ as discussed above, and for $i \in S_1$ there is a unique $j \in S_0$ and $r\in \{1,...,k(i,j)\}$ such that $p_{ijr}(m,\lambda,\tau)\uparrow 1$ as $\tau \downarrow 0$; $j$ and $r$ are determined by forcing onset to occur, which again agrees with $X_n$.  Defining $S_2 = S\setminus (S_0\cup S_1)$, for $i\in S_2$ each $p(i,j,r)$ converges to some number $a(i,j,r)$ such that $\sum_{j \in S,r=1,...,k(i,j)}a(i,j,r)=1$ and such that with respect to $a(i,j,r)$, $S_0\cup S_1$ is accessible from $S_2$; since for $\tau=0$, $S_0\cup S_1$ is invariant, we can define $p(i,j,r)(m,\lambda,0)=a(i,j,r)$ without affecting $X_n$, so convergence of the $p(i,j,r)$ is proved, \& we are done.\\

\subsection{Upper bound on $\lambda^+$ as $\tau\rightarrow\infty$}\label{secconcat}
We use the same approach as in the case $\tau\rightarrow 0$, and begin by describing the construction of \cite{supercrit} in somewhat better detail.\\

As mentioned in Section \ref{seccat}, the idea of \cite{supercrit} is that given $p<1$, we can choose rectangles $R_{m,n}$ and initial data for the process so that the set of $(m,n) \in L$ such that there is an active point in $R_{m,n}$ dominates the cluster $C(0,0)$ in a $1$-dependent oriented percolation model with parameter $p$.  The specific event on $R_{m,n}$ that allows this is the existence of active paths going from the bottom centre of $R_{m,n}$ to the top left and top right, with the property that if, say, $R_{m,n}$ and $R_{m-1,n+1}$ both have these paths, then said paths can be concatenated to form a longer active path through both rectangles.\\

Our strategy is first to show that the construction of \cite{supercrit} applies to the limit process, so that if $\lambda>\lambda_c^{\infty}$ then for any $p<1$ we can choose $R_{m,n}$ for the limit process that have the desired active paths, in a $1$-dependent way, with probability $\geq p$ for each $(m,n)$.  Then, by defining a condition for initial data at the base of $R_{m,n}$, and an event on each rectangle $R_{m,n}$, that allow us to deal with the possibility of discrepancies as encountered in the proof of Theorem \ref{thmlimit} in Section \ref{seclimit}, we can show that with nearly the same probability, for $\tau$ large enough the SEIS process has the same active paths, and these paths can be concatenated.  First we address applicability of the construction of \cite{supercrit} to the limit process.  

\begin{lemma}\label{limcomp}
The comparison to oriented percolation given in \cite{supercrit} is valid for the limit process on $\mathbb{Z}$, i.e., if $\lambda>\lambda_c^{\infty}$ and $\epsilon>0$ then suitably rescaled, the limit process dominates oriented site percolation with parameter $p>1-\epsilon$.
\end{lemma}

\begin{proof}
First note that our definition of the critical value is the same as theirs, namely of survival with positive probability starting from a single infectious site.  As mentioned at the end of Section 2 of \cite{supercrit}, their construction is valid for a broader class of models they call ``nearest neighbour additive growth models'' that includes the limit process.  The key properties required are additivity, described in Section 4 of \cite{speed}, and the coupling property described in Lemmas 3.1 and 3.4 of \cite{speed}, both of which are easily verified to hold for the limit process.  Additivity means that for configurations $\zeta$ and $\zeta'$, defining $\zeta\vee\zeta'$ for each $x$ by $(\zeta\vee\zeta)(x) = \max (\zeta(x),\zeta'(x))$, then with respect to the coupling given by the graphical construction, if $\zeta_0'' = \zeta_0\vee\zeta_0'$ then $\zeta_t'' = \zeta_t\vee\zeta_t'$ for $t>0$.  The desired coupling property is that, if we let $\zeta_0^1(x) \equiv 1$, $\zeta_0^+(x) = \mathbf{1}(x \geq 0)$ and $\zeta_0^0(x) = \mathbf{1}(x=0)$ then for $l_t^+ = \inf\{x:\zeta_t^+(x) \neq 0$, $l_t^0 = \inf\{x:\zeta_t^0(x) \neq 0\}$ or $=\infty$ if $\zeta_t^0(x)=0$ for all $x$, and $r_t^0 = \inf\{x:\zeta_t^0(x) \neq 0\}$ or $=-\infty$ if that set is empty, it holds that $\zeta_t^0 = \zeta_t^+ \cap [l_t^0,r_t^0] = \zeta_t^1 \cap [l_t,r_t]$ i.e., on the interval $[l_t,r_t]$, $\zeta_t^0$ agrees with $\zeta_t^+$ and with $\zeta_t^1$, and if $r_t^0>-\infty$ then $l_t^0 = l_t^+$.  The coupling property can be checked by examining transitions occurring near the endpoints $l_t^0,r_t^0$.
\end{proof}

Next we introduce two slightly different definitions of active path for the SEIS process that will be helpful.  An active path literally implies that all points along that path are active, while a potentially active path will, under some additional conditions on the state of the process near the base of the path and on the surrounding substructure, also be active.

\begin{definition}\label{2defactive}
For the SEIS process $\eta_t$, given $\eta_0$ and $\tau>0$, a path as defined in Definition \ref{defpath} is \emph{active} if for $i=1,...,m-1$, $\eta_{t_i}(x_i)=2$, $\eta_{t_i}(x_{i+1})=1$, there is a $\leftrightarrow$ label at $h_i$, and for $i=1,...,m$, $\eta_t(x_i)\neq 0$ for $t \in (t_{i-1},t_i)$.
\end{definition}

\begin{definition}
For the SEIS process and a path $\gamma=(v_1,h_1,...,v_{m-1},h_{m-1},v_m)$ as defined in Definition \ref{defpath} say that $\gamma$ is \emph{potentially active} if
\begin{enumerate}
\item for $i=1,...,m-1$, there is a $\leftrightarrow$ label at $h_i$,
\item for $i=1,...,m-1$ there is a $\star$ label at a point $(x_i,t) \in v_i = \{x_i\}\times (t_{i-1},t_i)$ called the \emph{activating} label, which is such that there are no $\star$ labels in $\{x_i\}\times (t_{i-1},t)$ and no $\times$ labels in $\{x_i\}\times(t,t_i)$, and
\item for any $\star$ label at a point $(x_m,t) \in v_m = \{x_m\}\times(t_{m-1},t_m)$ there are no $\times$ labels in $\{x_m\}\times (t,t_m)$.
\end{enumerate}
\end{definition}

Next we generalize the condition on the initial configuration given in Theorem \ref{thmlimit}, and the event described in Lemma \ref{onset}, to a larger class of spacetime sets.  Together these are probably the simplest conditions under which the SEIS process is well behaved.

\begin{definition}\label{defoo}
Let $R \subset \mathcal{S}$ be a set in spacetime which is the closure of an open set, and define the \emph{base} of $R$ as
\begin{equation*}
\base(R) = \{(x,t) \in R: (x,t-\epsilon) \notin R \textrm{ for all small enough }\epsilon>0\}
\end{equation*}
For the rescaled SEIS process, say that $R$ is \emph{onset-ordered} if the event described in Lemma \ref{onset} holds on $R$ i.e. if there is a $\star$ label at $(x,t) \in R$ and $\{x\} \times (t,t') \subset R$ for $t<t'$ then for any $\star$ label at a point $(y,s)$, $s \in (t,t')$ there is a $\times$ label at $\{x\}\times s'$ for some $s' \in (t,s)$.  Say that $\eta_t$ is \emph{good} for $R$ if $\eta_t(x)\neq 2$ for $(x,t) \in \base(R)$.
\end{definition}

The following result allows us to promote a potentially active path to an active path, when the above-stated conditions hold.

\begin{lemma}\label{potact}
Let $R$ be as in Definition \ref{defoo}.  If $R$ is onset-ordered and $\eta_t$ is good for $R$ then
\begin{enumerate}
\item for each $t\geq 0$, $\eta_t(x)=2$ for at most one $x$ in the set $\{x \in V:(x,t) \in R\}$, and
\item a potentially active path $\gamma$ with base $(y,s)$ satisfying $\eta_s(y)=1$ is active in the sense of Definition \ref{2defactive}.
\end{enumerate}
\end{lemma}
\begin{proof}
If $\eta_t(x)=2$ for $(x,t) \in R$ then by tracing back in time along the fiber $\{x\}\times \mathbb{R}^+$, there is a point $(x,s)$, $s \leq t$ such that if $s<t$ then $\{x\}\times (s,t) \in R$ and there are no $\times$ labels in $\{x\}\times (s,t)$, and such that either
\begin{enumerate}
\item $(x,s) \in \intr(R)$, the interior of $R$, $\eta_{s-\epsilon}(x)=1$ for all small enough $\epsilon>0$ and there is a $\star$ label at $(x,s)$ or
\item $(x,s)\in \base(R)$ and $\eta_s(x)=2$
\end{enumerate}
If $\eta_t$ is good for $R$ then the second case does not occur.  Suppose $\eta_t(x)=\eta_t(y)=2$ for some $x \neq y$, and let $s_x,s_y$ be the times such that $(x,s_x)$ and $(y,s_y)$ have the property stated above; we may assume that $s_x \leq s_y$, and since labels almost surely do not occur simultaneously, that $s_x<s_y$.  But then there is an interval $\{x\}\times (s_x,t) \in R$ and a $\star$ label at a point $(y,s_y) \in R$, $s_x<s_y<t$, such that there are no $\times$ labels in the interval $\{x\}\times (s_x,s_y)$, so $R$ is not onset-ordered.  The second statement follows from the first statement, the definition of potentially active path, and the transition rules.
\end{proof}

The next result allows us, under fairly mild conditions, to concatenate a collection of potentially active paths into a longer potentially active path.  The words before, after, starts, etc. are with respect to the order of events in time.

\begin{lemma}\label{concat}
Let $R$ be as in Definition \ref{defoo} and suppose $\gamma_1,...,\gamma_k$ are potentially active paths on $R$ with respective bases $(x_1,t_1),...,(x_k,t_k)$ satisfying
\begin{enumerate}
\item $t_i<t_{i+1}$, $i=1,...,k-1$ i.e., $\gamma_i$ starts before $\gamma_{i+1}$, and
\item $(x_{i+1},t_{i+1})$ does not intersect $\gamma_i$, $i=1,...,k-1$ i.e., $\gamma_{i+1}$ does not start somewhere on $\gamma_i$,
\end{enumerate}
and with intersection points $(y_i,s_i) \in \gamma_i \cap \gamma_{i+1},\,\,i=1,...,k-1$, $y_i\in V$, satisfying
\begin{enumerate}
\item $s_i<s_{i+1}$, $i=1,...,k-2$,
\item $y_i\neq y_{i+1}$, $i=1,...,k-2$, and
\item $y_{k-1}$ is not on the last vertical segment of $\gamma_k$.
\end{enumerate}
If $R$ is onset-ordered then the path $\gamma$ obtained by concatenating $\gamma_1,...,\gamma_k$ through the points $(y_i,s_i),\,\, i=1,..,k-1$ is a potentially active path.
\end{lemma}

\begin{proof}
For $i \in 1,..,k-1$ let $u_i,v_i$ be the vertical segments in $\gamma_i,\gamma_{i+1}$ respectively that satisfy $(y_i,s_i) \in u_i\cap v_i$ and let $w_i$ be the vertical segment in $\gamma$ containing $(y_i,s_i)$.  For $i=1,...,k-1$, $w_i$ is not the last vertical segment in $\gamma$, $w_i \subset u_i \cup v_i$ and more precisely, $w_i=\{y_i\}\times [\min u_i,\max v_i]$ where $\min u_i$ is the lowest point (in time) on $u_i$ and $\max v_i$ is the highest point on $v_i$; for $i=1,...,k-2$ this follows from the assumption $y_i \neq y_{i+1}$ and for $i=k-1$ it follows from the assumption $y_{k-1}$ is not on the last verical segment of $\gamma_k$.  Except for the segments $w_i$, $i=1,...,k-1$, the segments of $\gamma$ correspond to segments in the paths $\gamma_1,...,\gamma_k$, so it suffices to check for $i=1,..,k-1$ that on $w_i$ there is an activating label i.e. a $\star$ label with no $\times$ labels on $w_i$ after it and no $\star$ labels on $w_i$ before it; we consider separately the cases $\min v_i < \min u_i$, $\min u_i = \min v_i$ and $\min u_i < \min v_i$.\\

If $\min v_i < \min u_i$ then since the base of $\gamma_i$ comes before the base of $\gamma_{i+1}$ there is a vertical segment $u_{i-1}$ in $\gamma_i$ that precedes $u_i$.  Since neither $u_{i-1}$ nor $v_i$ are the last vertical segments on their respective paths, there is a $\star$ label on $u_{i-1}$ with no $\times$ labels on $u_{i-1}$ after it, and a $\star$ label $\ell$ on $v_i$ with the analogous property.  Since $R$ is onset-ordered, it follows that there are no $\star$ labels between the time $\ell$ occurs and $\max v_i$, so $\ell$ cannot occur before $\min u_i$.  Thus $\ell$ occurs after $\min u_i$ and lies on $w_i$.  It is easy to check that $\ell$ is an activating label for $w_i$.\\

If $\min u_i = \min v_i$ then $w_i=v_i$ and $w_i$ inherits the activating label from $v_i$.  If $\min u_i < \min v_i$, then since $\gamma_{i+1}$ does not start on $\gamma_i$, there is a vertical segment $v_{i-1}$ preceding $v_i$.  Any $\star$ label on $u_i$ has no $\times$ labels on $u_i$ after it.  From the same argument as in the previous paragraph with $v_{i-1}$ playing the role of $u_{i-1}$ and $u_i$ playing the role of $v_i$, it follows that there are no $\star$ labels on $u_i\setminus v_i$, and so $w_i$ inherits the activating label from $v_i$.
\end{proof}

After assembling the above ideas, we can prove in a straightforward way the domination of a 1-dependent oriented percolation process.  Once we have shown this is true, the rest of the proof of the upper bound on $\lambda^+$ proceeds in the same way as when $\tau\rightarrow 0$, which we leave to the reader.
\begin{lemma}
For fixed $\lambda>\lambda_c^{\infty}$ and $\epsilon>0$, for $\tau$ large enough the SEIS process with parameters $\lambda,\tau$ dominates a 1-dependent oriented percolation process with parameter $p\geq 1-\epsilon$.
\end{lemma}

\begin{proof}
First, construct the limit process $\zeta_t$ and apply the construction of \cite{supercrit} to obtain rectangles $R_{m,n} = (mK,nT)+[-J,J]\times[0,1.2T]$ such that the corresponding events $A_{m,n}$ have $\mathbb{P}(A_{m,n})\geq 1-\epsilon/3$.  As specified more precisely in \cite{supercrit}, each event $A_{m,n}$ corresponds to the existence of some active paths in $R_{m,n}$ going from the base of $R_{m,n}$ to some locations in the top of $R_{m,n}$, such that if $A_{m_i,n_i}=1$ for each $(m_i,n_i)$ in a path $(0,0)=(m_1,n_1),...,(m_k,n_k)$, then the resulting paths intersect to produce an active path from some point $(y,0)$, $y\in [-J,J]$ to a point $(x,t)$ in $R_{m_k,n_k}$.\\

Independently, construct the rescaled SEIS process $\eta_t$ and superimpose onto it the rectangles $R_{m,n}$ from the last paragraph.  For $\delta>0$, define the region $P_{m,n} = (mK,NT)+[-J,J]\times[-\delta,0]$ that lies just below $R_{m,n}$ in spacetime, and say that $P_{m,n}$ is good for $R_{m,n}$ if, depending only on the labelling on $P_{m,n}$, for all possible configurations $\eta_{nT-\delta}$, $\eta_{nT}$ is good for $R_{m,n}$.  It is easy to check that for $\tau$ large enough and $\delta>0$ small enough, with probability $\geq 1-\epsilon/3$, $P_{m,n}$ is good for $R_{m,n}$, due to the resulting plenitude of $\times$ labels and lack of $\star$ labels on $P_{m,n}$.  Suppose $\eta_t$ is good for $R_{m,n}$, then construct a new copy $\zeta_t^{(m,n)}$ of the limit process only on $R_{m,n}$ using the method of Section \ref{seclimit} with initial configuration $\{\eta_{nT}(x):(x,nT) \in R_{m,n}\}$, and note that when $R_{m,n}$ is onset-ordered, to each active path in $\zeta_t^{(m,n)}$ corresponds a potentially active path for $\eta_t$ on $R_{m,n}$ in the obvious way.  Let $Q_{m,n} = R_{m,n}\cup P_{m,n}$ and define new events $B_{m,n}$ on $Q_{m,n}$ by
\begin{equation*}
B_{m,n} = \{P_{m,n}\textrm{ is good for }R_{m,n},\,\,R_{m,n}\textrm{ is onset-ordered and }A_{m,n}\textrm{ holds for }\zeta_t^{(m,n)}\}
\end{equation*}
If $\tau$ is large enough then $\mathbb{P}(R_{m,n}\textrm{ is onset-ordered }) \geq 1-\epsilon/3$, and if $\delta>0$ is small enough then the oriented percolation model defined by the events $A_{m,n}'$ is 1-dependent, so for $\tau$ large enough and $\delta>0$ small enough, $A_{m,n}'$ is 1-dependent and $P(A_{m,n}') \geq 1-\epsilon$.\\

If $\eta_t$ is good for a finite collection of sets then it is good for the union of those sets.  Moreover, if $(m_1,n_1),...,(m_k,n_k)$ is a lattice path, $R_{m_1,n_1},...,R_{m_k,n_k}$ are onset-ordered and $A_{m_1,n_1},...,A_{m_k,n_k}$ holds for $\zeta_t^{(m_1,n_1)},...,\zeta_t^{(m_k,n_k)}$, then it is easy to see from the geometry of the construction in \cite{supercrit} that the potentially active paths for $\eta_t$ in the rectangles $R_{m_1,n_1},...,R_{m_k,n_k}$ corresponding to the relevant active paths for $\zeta_t^{(m_1,n_1)},...,\zeta_t^{(m_k,n_k)}$ satisfy the conditions of Lemma \ref{concat}.  Thus if $B_{m_1,n_1}\cap...\cap B_{m_k,n_k}$ holds for a path $(m_1,n_1),...,(m_k,n_k)$ then in each of the rectangles $R_{m_1,n_1},...,R_{m_k,n_k}$ there is a potentially active path for $\eta_t$, such that the concatenation of those paths is a potentially active path through the union $R_{m_1,n_1}\cup....\cup R_{m_k,n_k}$ which, since $\eta_t$ is good for the union, is an active path for $\eta_t$, and the desired stochastic domination follows.  
\end{proof}

\subsection{Lower bound on $\lambda^-$ as $\tau\rightarrow\infty$}
An argument given in \cite{basic} shows that the coupling property described in the proof of Lemma \ref{limcomp} implies that $\lambda_c^{\infty} = \sup\{\lambda:\alpha > 0\}$, where $\alpha = \lim_{t\rightarrow\infty}l_t^+/t$; note that our definition of critical value agrees with the one in \cite{basic}, namely as survival with positive probability, started from a single infectious site.  With this fact, the proof given in \cite{ziezold} that $\lambda_m \uparrow \lambda_c$ applies without modification to the limit process.  We then proceed in the same way as when $\tau\rightarrow 0$.  In this case we shall have $S$ as before, $S_0 = \{\eta \in \{0,1,2,3\}:\eta(x) \in \{0,1\}, x=1,...,m+1\}$ and $S_1 = \{\eta \in \{0,1,2,3\}:\eta(x)=2 \textrm{ for some }x \textrm{ and }\eta(y) \notin \{2,3\}, y\neq x\}$, and modify $X_n$ when $\tau=\infty$ to include intermediate transitions through the infectious state and propagation before recovery.  Since the proof is analogous, the details are omitted.

\section*{Acknowledgements}
The author's research is partly supported by an NSERC PGSD2 Scholarship.
\bibliography{SEISbasics}
\bibliographystyle{plain}
\end{document}